\documentclass{article}

\usepackage{amssymb}
\usepackage{amsmath,amsthm}
\usepackage[dvips]{graphicx}
\usepackage{wrapfig}
\usepackage{bm}

\newtheorem{theorem}{Theorem}[section]
\newtheorem{corollary}[theorem]{Corollary}
\newtheorem{lemma}[theorem]{Lemma}

\theoremstyle{definition}

\theoremstyle{remark}

\begin{document}

\title{Unknottability of spatial graphs by region crossing changes}

\author{
Yukari Funakoshi \thanks{Faculty of Education, Gifu Shotoku Gakuen University, 1-1 Takakuwa-nishi, Yanaizu-shi, Gifu, 501-6194, Japan. Email: yukarifunakoshi@gifu.shotoku.ac.jp}
\and 
Kenta Noguchi \thanks{Department of Information Sciences, Tokyo University of Science, 2641 Yamazaki, Noda, Chiba, 268-8510, Japan. Email: noguchi@rs.tus.ac.jp}
\and 
Ayaka Shimizu \thanks{Department of General Education, National Institute of Technology (KOSEN), Gunma College, 580 Toriba, Maebashi-shi, Gunma, 371-8530, Japan. Email: shimizu@gunma-ct.ac.jp}
}
\date{\today}

\maketitle

\begin{abstract}
A region crossing change is a local transformation on spatial graph diagrams switching the over/under relations at all the crossings on the boundary of a region. 
In this paper, we show that a spatial graph of a planar graph is unknottable by region crossing changes if and only if the spatial graph is non-Eulerian or is Eulerian and proper. 
\end{abstract}

\section{Introduction}

A {\it knot} is an embedding of a circle to a three-sphere. 
A {\it link of $r$ components} is an embedding of $r$ circles to a three-sphere. 
A {\it spatial graph of $r$ components} is an embedding of $r$ connected graphs to a three-sphere. 
By regarding a circle to be a graph without vertices, we assume that knots and links belong to spatial graphs. 
Each spatial graph $S$ is represented by a {\it diagram} on a two-sphere, a projection of $S$ to a two-sphere with over/under information at each intersection, where each intersection is a double point of edges and called a {\it crossing}. 
It is well-known that two diagrams represent the same spatial graph if and only if one of them can be transformed into the other by a finite number of the Reidemeister moves RI to RV shown in Figure \ref{reidemeister} (\cite{kauffman}). 
\begin{figure}[ht]
\begin{center}
\includegraphics[width=90mm]{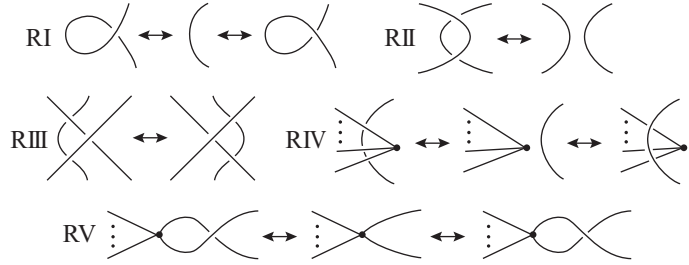}
\caption{Reidemeister moves. }
\label{reidemeister}
\end{center}
\end{figure}
A {\it self-crossing} (resp. {\it non-self-crossing}) on a diagram is a crossing between edges of the same (resp. different) component. 
A {\it planar graph} is a graph which can be embedded to a two-sphere without creating crossings. 
A spatial graph $S$ of a planar graph is {\it unknotted} if $S$ has a diagram which has no crossings. 
A diagram $D$ of a spatial graph is {\it unknotted} if $D$ represents an unknotted spatial graph. 
A spatial graph $S$ is {\it completely splitted} if $S$ has a diagram which has no non-self-crossings. 
A graph $G$ is {\it Eulerian} if the degree of every vertex of $G$ is even. 
A spatial graph $S$ is {\it Eulerian} if $S$ is an embedding of an Eulerian graph. 
We assume that knots and links are Eulerian. 

Studies of local transformations have a key role in knot theory and spatial graph theory to measure a complexity of a spatial graph or to consider the relations or classifications of spatial graphs. 
For example, a {\it Delta move} is a local transformation on spatial graphs shown in Figure \ref{Delta}. 
It is shown in \cite{MN} that a Delta move is an unknotting operation for knots, i.e., we can unknot any knot by applying a finite number of Delta moves and Reidemeister moves on its diagram. 
On the other hand, a Delta move is not an unknotting operation for links and spatial graphs. 
Then the equivalent classes of links and spatial graphs on Delta moves are studied using and applying to other invariants, such as the Conway polynomial and the Wu invariant (\cite{MT, nakanishi, NO, okada, taniyama}). 
\begin{figure}[ht]
\begin{center}
\includegraphics[width=40mm]{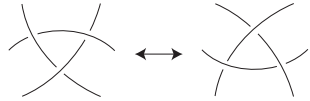}
\caption{A Delta move. }
\label{Delta}
\end{center}
\end{figure}

A {\it $\Delta _{13}$-move} is a local transformation on spatial graphs shown in Fig. \ref{d13} \cite{nakanishi2}. 
It is shown in \cite{nakanishi2} that a $\Delta _{13}$-move is an unknotting operation for knots, and is not for links. 
For spatial graphs, it is shown in \cite{ohyama} that a spatial graph $S$ of a planar graph can be unknotted by $\Delta _{13}$-moves if and only if $S$ is non-Eulerian or is Eulerian and proper. 
The definition of the properness is given in Section \ref{section-Eulerian}. 
\begin{figure}[ht]
\begin{center}
\includegraphics[width=40mm]{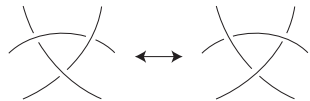}
\caption{A $\Delta _{13}$-move. }
\label{d13}
\end{center}
\end{figure}

A {\it region crossing change} is a local transformation on spatial graph diagrams which changes the over/under information at all the crossings on the boundary of a region. 
The following theorem is shown for knot diagrams
\footnote{
An alternative proof of Theorem \ref{shimizu-thm} is given in \cite{DR} using graph theory. 
}. 

\begin{theorem}[\cite{shimizu-rcc}]
Any diagram of a knot can be unknotted by region crossing changes. 
\label{shimizu-thm}
\end{theorem}

\noindent Note that it had already been shown in \cite{aida} that any knot has a diagram which can be transformed into an unknotted diagram by a single ``$n$-gon move'',  a kind of region crossing changes. 
The point of Theorem \ref{shimizu-thm} is that we can unknot any fixed diagram of a knot by region crossing changes without applying Reidemeister moves. 
For links, the following is shown. 

\begin{theorem}[\cite{cheng}]
Any diagram of a link $L$ can be unknotted by region crossing changes if and only if $L$ is proper. 
\label{cheng-thm}
\end{theorem}

\noindent The point of Theorem \ref{cheng-thm} is that the unknottability of link diagrams by region crossing changes depends only on the properness of a link itself. 
The following theorem is shown for spatial graphs of a connected planar graph. 

\begin{theorem}[\cite{HSS}]
Any diagram of a spatial graph of a connected planar graph can be unknotted by region crossing changes. 
\label{HSS-thm}
\end{theorem}

\noindent Theorem \ref{shimizu-thm} implies that a region crossing change is an unknotting operation for knot diagrams and Theorem \ref{cheng-thm} implies that it is not an unknotting operation for link diagrams. 
Again, the point of Theorems \ref{shimizu-thm}, \ref{cheng-thm} and \ref{HSS-thm} is that it does not depend on the choice of a diagram. 
In general, the unknottability by region crossing changes depends on the choice of a diagram of the spatial graph as pointed out in \cite{ST}. 
We define that a spatial graph $S$ is {\it unknottable (resp. completely splittable) by region crossing changes} if $S$ has a diagram which can be unknotted (resp. completely splitted) by region crossing changes, where applying Reidemeister moves is not allowed during region crossing changes. 
Note that any spatial graph of a planar graph is unknottable by (the classical) crossing changes. 
In this paper, we show the following theorems as a generalization of Theorems \ref{shimizu-thm}, \ref{cheng-thm} and \ref{HSS-thm}. 

\begin{theorem}
A spatial graph $S$ of a planar graph is unknottable by region crossing changes if and only if $S$ is non-Eulerian or is Eulerian and proper. 
\label{thm-unknottable}
\end{theorem}

\begin{theorem}
A spatial graph $S$ is completely splittable by region crossing changes if and only if $S$ is non-Eulerian or is Eulerian and proper.\footnote{
Any spatial graph of a non-Eulerian graph is completely splittable by region crossing changes. 
This means that the splitness by region crossing changes is {\it intrinsic} (see \cite{SGT}) to non-Eulerian graphs. 
On the other hand, since it depends on the way of embedding, splitness by region crossing changes is not intrinsic to Eulerian graphs.}
\label{thm-splittable}
\end{theorem}

\noindent The rest of the paper is organized as follows: 
In Section \ref{section-non-Eulerian}, we consider non-Eulerian spatial graphs. 
In Section \ref{section-Eulerian}, we consider Eulerian spatial graphs. 
In Section 4, we prove Theorems \ref{thm-unknottable} and \ref{thm-splittable}.

\section{Non-Eulerian spatial graphs}
\label{section-non-Eulerian}

In this section we consider non-Eulerian spatial graphs and show the following lemma:  

\begin{lemma}
Let $S$ be a non-Eulerian spatial graph. 
Let $D$ be a diagram of $S$, and let $D'$ be a diagram which is obtained from $D$ by some crossing changes. 
There exists a diagram $E$ of $S$ such that $E$ can be transformed into a diagram representing the same spatial graph to $D'$ by region crossing changes. 
\label{lemma12}
\end{lemma}

For a spatial graph diagram $D$, a crossing $c$, a vertex $v$ and a path $P$ connecting $c$ and $v$, we define the following transformation and denote it by $cPv$. 
Take an (over or under) arc $\alpha$ of $c$ which does not belong to $P$. 
Stretch $\alpha$ along $P$ to pass $v$ as shown in Figure \ref{alongP} (cf. \cite{ozawa-edge}). 
\begin{figure}[ht]
\begin{center}
\includegraphics[width=70mm]{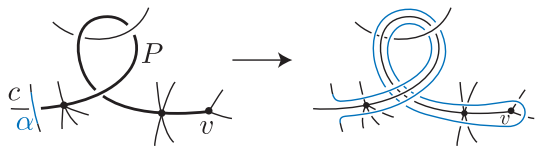}
\caption{The transformation $cPv$. The spur is colored blue on the right figure. }
\label{alongP}
\end{center}
\end{figure}
Note that $cPv$ is realized by Reidemeister moves. 
We call the stretched $\alpha$ the {\it spur} of $cPv$. 
Note that the over/under relationship for the spur to all the edges around the vertices on $P$ are the same to that for $\alpha$ to $P$. 
To prove Lemma \ref{lemma12}, we need the following lemma. 

\begin{lemma}
Let $D$ be a spatial graph diagram, and let $c$ be a crossing and $v$ be a vertex of odd degree, connected by a path $P$, where $P$ has no vertices of odd degree except $v$. 
Let $D'$ be a diagram obtained from $D$ by a crossing change at $c$. 
Let $D''$ be a diagram obtained from $D$ by $cPv$. 
Then $D''$ can be transformed into a diagram representing the same spatial graph to $D'$ by region crossing changes. 
\label{lem-even-vertex}
\end{lemma}

\begin{proof}
Let $u_1, u_2, \dots ,u_k=v$ be the vertices on $P$ in order from the side of $c$ along $P$. 
We locally consider edges not on $P$ which are incident to a vertex $u_i$; say $e_1, e_2, \dots ,e_l$, in cyclic order along the spur, as shown in Figure \ref{spur-checker}. 
We remark that $l$ is an even number because each vertex $u_i$ has locally even number of edges which intersect the spur. 
Take the regions in the spur along $P$ between $e_i$ and $e_{i+1}$, where we cancel the region which we encounter twice at a self-crossing of $P$. 
We call the set of the regions $Q_i$. 
\begin{figure}[ht]
\begin{center}
\includegraphics[width=110mm]{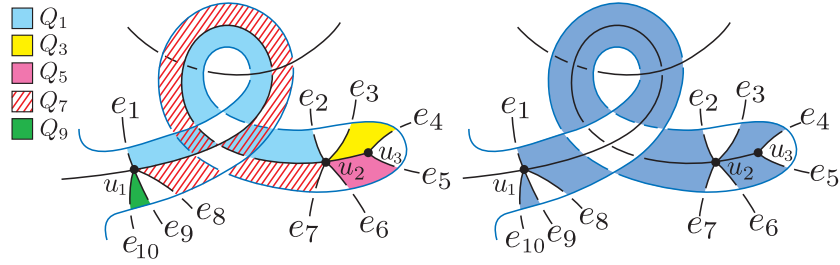}
\caption{Regions in $Q_1, Q_3, Q_5, Q_7, Q_9$ are color-coded on the left. The regions of $R(cPv)$ are colored blue on the right. }
\label{spur-checker}
\end{center}
\end{figure}
Let $R(cPv)$ be the symmetric difference of $Q_1, Q_3, Q_5, \dots ,Q_{l-1}$. 
By applying region crossing changes at all the regions in $R(cPv)$, the over/under relationship around all the vertices in $P$ is changed, and any other crossing is unchanged. 
Hence, we can shrink the spur back through the other side, and obtain $D'$. 
\end{proof}

\phantom{x}
\noindent {\bf Proof of Lemma \ref{lemma12}} \ 
Let $G$ be a connected non-Eulerian graph. 
Let $S$ be a spatial graph consisting of some graphs including $G$, and let $D$ be a diagram of $S$. 
Note that by the Handshaking Lemma, $G$ has two or more even number of vertices of odd degree. 
Let $V_{\mathrm{odd}}(G)$ be the set of all the vertices of odd degree of $G$. 
Let $B= \{ b_1, b_2, \dots ,b_k \}$ be a set of some crossings of $D$ which are on $G$, where a crossing on $G$ means a self-crossing or non-self-crossing which belongs to the diagram of $G$. 
Let $C= \{ c_1, c_2, \dots ,c_l \}$ be a set of some crossings of $D$ which are not on $G$. 
Let $D'$ be a diagram which is obtained from $D$ by crossing changes at all the crossings in $B$ and $C$. 
We show that we can retake a diagram $E$ of $S$ from $D$ so that $E$ can be transformed into a diagram representing the same spatial graph to $D'$ by region crossing changes. \\
(a) Let $F_0 =D$. \\
(b) Take a path $P_i$ on $F_{i-1}$ which connects $b_i$ and one of the vertices $v_i$ in $V_{\mathrm{odd}}(G)$ so that $P_i$ does not include any other vertices in $V_{\mathrm{odd}}(G)$, where $v_i$ and $v_j$ may be the same vertex for $i \neq j$. 
Apply $b_i P_i v_i$. \\
Repeat the procedure (b) from $i=1$ to $i=k$, and let $F=F_k$. 

For $F$, take the symmetric difference $R^F$ of $R(b_i P_i v_i )$ for $i=1, 2, \dots , k$. 
Note that some regions in $R(b_i P_i v_i)$ may be divided by $b_j P_j v_j$ ($i<j$). 
In that case, retake all the corresponding regions as $R(b_i P_i v_i)$. 
Thus, by Lemma \ref{lem-even-vertex}, all the over/under relationship around the vertices of $P_i$ will be changed for every spur of $b_i P_i v_i$ if we apply region crossing changes at the regions of $R^F$. 

\phantom{x}

Then for $F$, take two vertices $v_1$ and $v_2$ from $V_{\mathrm{odd}}(G)$ such that $v_1$ and $v_2$ are connected by a path $P$ which does not have any other vertices of odd degree. \\
(c) Let $E_0 =F$. \\
(d) Take an arc $\alpha$ in $P$, and stretch $\alpha$ to $c_i$ going through the over side of other edges as shown in the middle figure of Figure \ref{face}. 
Four crossings are created and we call the crossings on the ends $c_i^1$ and $c_i^2$. 
Apply $c_i^1 P_i^1 v_1$ and $c_i^2 P_i^2 v_2$, where $P_i^j$ is the path connecting $c_i^j$ and $v_j$, and we remark that there are no vertices of odd degree in $P_i^j$ except $v_j$. 
We call the region adjacent to $c_i$ which is created by the stretch of $\alpha$ in the above procedure $R_i$. 
Note that if we apply region crossing changes at $R_i$, $R(c_i ^1 P_i ^1 v_1)$ and $R(c_i ^2 P_i ^2 v_2 )$, then the over/under relation will be changed at the two spurs and $R_i$, and then if we shrink back the two spurs and $\alpha$, a crossing change at $c_i$ will be realized. 
Repeat the procedure (d) from $i=1$ to $i=l$, and let $E=E_l$. 
\begin{figure}[ht]
\begin{center}
\includegraphics[width=100mm]{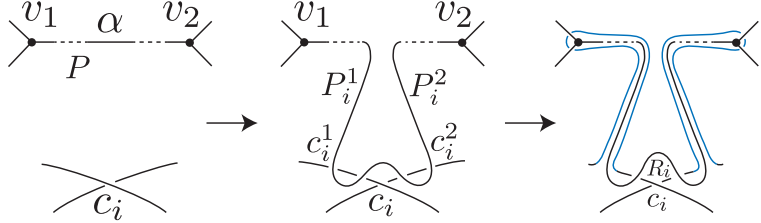}
\caption{The procedure (d). }
\label{face}
\end{center}
\end{figure}

For $E$, take the symmetric difference of $R^F$ and $R_i$ and $R(c_i^j P_i^j v_j )$ for all $i=1, 2, \dots , l$ and $j=1, 2$. 
Note that some regions in $R^F$, $R(c_i^j P_i^j v_j )$ and $R_i$ may be divided by the procedure for $c_h$ ($i<h$). 
In that case, retake all the corresponding regions. 
Thus, by applying region crossing changes and shrinking, we obtain $D'$. 
Note that the above transformations do not influence each other. \\
\hfill$\Box$

\section{Eulerian spatial graphs}
\label{section-Eulerian}

In this section, we review the study of region crossing changes for links and consider Eulerian spatial graphs. 

\subsection{Linking number of links}

Let $L$ be an oriented link of $r$ components $K_1, K_2, \dots , K_r$. 
Let $D$ be a diagram of $L$. 
The {\it linking number $lk(K_i , K_j)$ between $K_i$ and $K_j$} is the value of half the sum of the signs (see Figure \ref{crossing-sign}) for all the crossings between $K_i$ and $K_j$ in $D$. 
\begin{figure}[ht]
\begin{center}
\includegraphics[width=30mm]{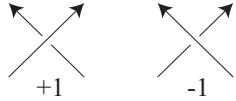}
\caption{The sign of a crossing. }
\label{crossing-sign}
\end{center}
\end{figure}
The value $lk(K_i, K_j)$ is an integer because the number of non-self-crossings of two components is an even number. 
It is well-known that $lk(K_i, K_j)$ is a link invariant since it is unchanged over Reidemeister moves RI, RI\vspace{-0.5pt}I and RI\vspace{-0.5pt}I\vspace{-0.5pt}I (see, for example, \cite{adams}). 
A link $L$ is {\it proper} if the value $\sum _{j \neq i} lk(K_i, K_j)$ is even for all $i \in \{1, 2, \dots ,r \}$ with an orientation. 
The properness does not depend on the choice of orientation because we have $lk(K_i , -K_j)=-lk(K_i , K_j)$, where $-K_j$ means $K_j$ with orientation reversed. 
Since the number of crossings between a component $K_i$ and the other components at the boundary of each region is an even number, the following holds: 

\begin{lemma}[\cite{cheng, CG}]
The properness of a link is preserved over region crossing changes. 
\label{rcc-proper}
\end{lemma}

\noindent The following lemma is shown in \cite{cheng}. 

\begin{lemma}[\cite{cheng}]
Let $D$ be a diagram of a link. 
Take $n$ knot components $D_1 , D_2 , \dots , D_n$ such that $D_i$ and $D_j$ have crossings for $| i-j | =1$ or $| i-j | =n-1$. 
Let $C$ be a set of $n$ crossings $c_1 , c_2 , \dots , c_n$ where $c_i$ is of $D_i$ and $D_{i+1}$ with $D_{n+1}=D_1$. 
Then the crossing changes at the $n$ crossings in $C$ can be realized by region crossing changes on $D$. 
\label{cycle-cc}
\end{lemma}

\noindent In particular, the following lemma holds. 

\begin{lemma}[\cite{cheng, CG}]
Let $D$ be a diagram of a link. 
Crossing changes on $D$ at any pair of crossings between two knot components, say $D_i$ and $D_j$, can be realized by region crossing changes for any $i$ and $j$. 
\label{pair-cc}
\end{lemma}

\subsection{Linking number of Eulerian spatial graphs}

In this subsection, we give the definition of the linking number for Eulerian spatial graphs, which is equivalent to the definition given in \cite{ohyama}. 
Let $G$ be a graph with an orientation on each edge. 
For a vertex $v$, the {\it indegree} (resp. {\it outdegree}) is the number of incident edges whose orientation is incoming to (resp. outgoing from) $v$. 
Let $S$ be an Eulerian spatial graph consisting of connected graphs $G_1, G_2, \dots ,G_r$.  
Give an orientation $O_i$ to $G_i$ so that the indegree equals the outdegree at each vertex of $G_i$. 
We call the orientation an {\it Eulerian orientation}. 
Note that we can take an Eulerian orientation for every Eulerian graph since it has an ``Eulerian circuit''. 

Unless otherwise stated, an oriented Eulerian spatial graph means an Eulerian spatial graph with an Eulerian orientation in this paper. 
We define the linking number for oriented Eulerian spatial graphs. 
Let $D$ be a diagram of an oriented Eulerian spatial graph $S=S_1 \cup S_2 \cup \dots \cup S_r$. 
The {\it linking number} $lk(S_i , S_j)$ between $S_i$ and $S_j$ is the value of half the sum of the signs for all the crossings between $S_i$ and $S_j$ in $D$. 
The value of $lk(S_i, S_j)$ is an integer since we can confirm that the number of crossings between $S_i$ and $S_j$ is an even number by considering their Eulerian circuits and assuming them a link. 
The value $lk(S_i , S_j)$ is preserved over Reidemeister moves RI, RI\vspace{-0.5pt}I and RI\vspace{-0.5pt}I\vspace{-0.5pt}I as well as for links. 
For RIV, the value is also preserved because the number of positive crossings and negative crossings are the same around a vertex which is applied an RIV. 
For RV, the value is unchanged because there are no change for non-self crossings. 
Hence $lk(S_i , S_j)$ is an invariant for oriented Eulerian spatial graphs. 
Moreover, we have the following: 

\begin{lemma}
The parity of $lk(S_i, S_j)$ is an invariant for (unoriented) Eulerian spatial graphs. 
\label{parity-lk-g}
\end{lemma}

\begin{proof}
Let $lk(S_i, S_j)$ be the linking number between $S_i$ and $S_j$ with Eulerian orientations $O_i$ of $S_i$ and $O_j$ of $S_j$. 
Let $lk'(S_i, S_j)$ be that with Eulerian orientations ${O_i}'$ of $S_i$ and $O_j$ of $S_j$. 
Take the subgraph $H$ of $S_i$ by taking edges which have different orientations between $O_i$ and ${O_i}'$. 
Then $H$ is an Eulerian subgraph of $S_i$. 
Since $H$ and $S_j$ are Eulerian, the number of crossings between $H$ and $S_i$ is an even number. 
Hence the difference between $lk(S_i, S_j)$ and $lk'(S_i, S_j)$ is the half of some multiples of four, i.e., multiples of two. 
\end{proof}

\noindent An Eulerian spatial graph $S$ of $r$ components is {\it proper} if $\sum _{j \neq i} lk(S_i, S_j)$ is even for all $i \in \{1, 2, \dots ,r \}$ with an Eulerian orientation. 
We have the following corollary as well as Lemma \ref{rcc-proper}. 

\begin{corollary}
For diagrams of an Eulerian spatial graph, the properness is preserved over region crossing changes. 
\label{rcc-proper-g}
\end{corollary}

\begin{proof}
Fix an Eulerian orientation. 
Since the number of crossings between $S_i$ and the other components at the boundary of each region is an even number, the parity of $\sum _{i \neq j} lk(S_i , S_j)$ is unchanged by region crossing changes for each $i$. 
\end{proof}

Next, we introduce the warping degree for spatial graph diagrams and consider the relation to the linking number for Eulerian spatial graph diagrams. 
Let $D=D_1 \cup D_2 \cup \dots \cup D_r$ be a diagram of a spatial graph of $r$ components with an order. 
A {\it warping crossing point between $D_i$ and $D_j$} ($i<j$) is a crossing point such that $D_j$ is over than $D_i$. 
The {\it warping degree $w(D_i , D_j)$ between $D_i$ and $D_j$} ($i<j$) is the number of warping crossing points between $D_i$ and $D_j$. 
Note that a diagram $D$ with $w(D_i, D_j)=0$ for all $i<j$ represents a completely splitted spatial graph. 
The following holds. 
(See \cite{kawauchi} and \cite{shimizu-wd} for links.) 

\begin{lemma}
For any diagram $D=D_1 \cup D_2 \cup \dots \cup D_r$ of an oriented Eulerian spatial graph $S=S_1 \cup S_2 \cup \dots \cup S_r$, 
$w(D_i , D_j) \equiv lk(S_i , S_j) \pmod 2$ holds for each $i<j$, where each $D_k$ corresponds to $S_k$. 
\label{ld-lk}
\end{lemma}

\begin{proof}
If $w(D_i , D_j)=k$, apply crossing changes at all the warping crossing points between $D_i$ and $D_j$, and let $D^0=D^0 _1 \cup D^0 _2 \cup \dots \cup D^0 _r$ be the obtained diagram, where $D^0 _i$ corresponds to $D_i$. 
Since $w(D^0 _i , D^0 _j)=0$, the two components represented by $D^0 _i$ and $D^0 _j$ are split and hence the linking number is zero. 
This implies that $lk(S_i , S_j) \equiv k \pmod 2$ because the value of the linking number is changed by one by each single crossing change. 
Hence $w(D_i , D_j) \equiv lk(S_i , S_j) \pmod 2$ holds. 
\end{proof}

\subsection{Vertex splittings}

For a graph, a {\it vertex splitting at a vertex $v$ into $v'$ and $v''$} is the following transformation. 
Add two vertices $v'$ and $v''$, reattach the edges incident to $v$ to exactly one of $v'$ or $v''$, and remove $v$ (see Figure \ref{vertex-splitting}). 
We have the following: 

\begin{lemma}
Let $G$ be a connected Eulerian graph, and let $e_1, e_2$ and $e_3$ be edges of $G$ which is incident to a vertex $v$. 
Let $G_{12}$ (resp. $G_{23}$) be a graph obtained from $G$ by a vertex splitting of $v$ such that only $e_1$ and $e_2$ (resp. $e_2$ and $e_3$) are incident to $v'$. 
Either $G_{12}$ or $G_{23}$ is connected. 
\end{lemma}

\begin{proof}
Assume that $G_{12}$ is not connected. 
Then $G_{12}$ has a cycle $H$ including $e_1$ and $e_2$ since $G_{12}$ is also Eulerian. 
Then, there is a path connecting $e_1$ and $e_2$ in $G_{23}$ which corresponds to $H - v'$ in $G_{12}$. 
This implies $G_{23}$ is connected. 
\end{proof}
\begin{figure}[ht]
\begin{center}
\includegraphics[width=95mm]{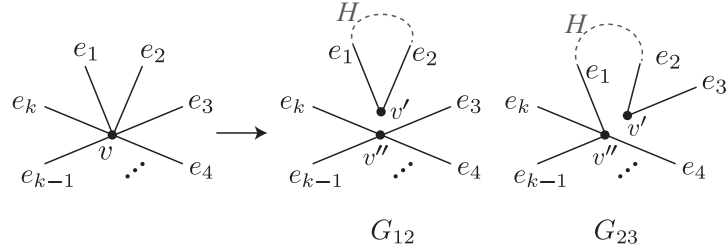}
\caption{Vertex splitting. }
\label{vertex-splitting}
\end{center}
\end{figure}

By repeating vertex splittings to $G$ keeping connected, and ignoring vertices of degree two, we obtain a closed curve without vertices. 
In terms of spatial graphs, we obtain a knot from a connected Eulerian spatial graph. 
We show the following: 

\begin{lemma}
Any diagram $D$ of a spatial graph of a connected Eulerian planar graph can be unknotted by region crossing changes. 
\label{conn-pl-E}
\end{lemma}

\begin{proof}
Let $C$ be a set of crossings of $D$ such that the crossing changes at all the crossings in $C$ make $D$ unknotted. 
Apply the above vertex splittings to $D$ to obtain a knot diagram $D^k$. 
Since any crossing change on any knot diagram can be realized by region crossing changes as shown in \cite{shimizu-rcc}, $D^k$ has a set $R$ of regions such that the region crossing changes realize the crossing changes at all the crossings in $C$. 
Apply region crossing changes to $D$ at the corresponding regions to $R$. 
\end{proof}

\noindent Note that Lemma \ref{conn-pl-E} is contained by Theorem 1.3 of \cite{HSS}. 
We have the following. 

\begin{lemma}
Let $D$ be a diagram of an Eulerian spatial graph, and let $D'$ be a link diagram obtained from $D$ by vertex splittings on each component. 
The follwoing (i) to (iv) are equivalent: \\
(i) $D$ is completely splittable by region crossing changes. \\
(ii) $D'$ is completely splittable by region crossing changes. \\
(iii) $D'$ is proper. \\
(iv) $D$ is proper. 
\label{lemma-four}
\end{lemma}

\begin{proof}
(i) $\Rightarrow$ (iv): The contraposition holds by Corollary \ref{rcc-proper-g}. \\
(ii) $\Leftrightarrow$ (iii): By Theorem \ref{cheng-thm}. \\
(iii) $\Leftrightarrow$ (iv): Give an orientation to $D'$, and give the same orientation to each edge of $D$. 
Then the orientation of $D$ is Eulerian, and we can see that the properness is the same for $D$ and $D'$. 
Note that even if $D'$ has extra crossings created by the vertex splittings, there are no influences because they are self-crossings. \\
(ii) $\Rightarrow$ (i): Let $R$ be a set of regions of $D'$ such that $D'$ is splittable by region crossing changes at the regions in $R$, and let $D'' =D''_1 \cup D'' _2 \cup \dots \cup D'' _r$ be the resulting of region crossing changes at $R$. 
Since the linking number is zero for each pair of components, the value of the warping degree $w(D'' _i , D'' _j)$ is even by Lemma \ref{ld-lk}. 
By Lemma \ref{pair-cc}, we can realize pairwise crossing changes at all the warping crossing points between $D'' _i$ and $D'' _j$ by region crossing changes at some regions, say $R_{i j}$. 
Apply region crossing changes to $D$ at the corresponding regions to the symmetric difference of $R$ and $R_{i j}$ for all $i < j$. 
Then we obtain a diagram with warping degree zero for any pair of components. 
Hence, $D$ is also completely splittable by region crossing changes. 
\end{proof}

\section{Proof of the main theorems}

In this section, we prove Theorems \ref{thm-unknottable} and \ref{thm-splittable}. 
For non-Eulerian spatial graphs, we have the following theorem by Lemma \ref{lemma12}. 

\begin{theorem}
Every non-Eulerian spatial graph is completely splittable by region crossing changes. 
\label{nE-splittable}
\end{theorem}

\begin{proof}
Let $S=S_1 \cup S_2 \cup \dots \cup S_r$ be a non-Eulerian spatial graph of $r$ components. 
Let $D$ be a diagram of $S$. 
Take a set $C$ of all the non-self-crossings between $S_i$ and $S_j$ such that $S_j$ is over than $S_i$ for all $i<j$. 
By Lemma \ref{lemma12}, $D$ can be transformed into a suitable diagram to change all the crossings in $C$ by region crossing changes. 
\end{proof}

\noindent Similarly, we have the following theorem. 

\begin{theorem}
Every spatial graph of a non-Eulerian planar graph is unknottable by region crossing changes. 
\label{nE-unknottable}
\end{theorem}

\begin{proof}
Let $D$ be a diagram of a spatial graph $S$ of a non-Eulerian planar graph. 
Since $S$ is an embedding of a planar graph, we can transform $D$ into an unknotted diagram by some crossing changes. 
By Lemma \ref{lemma12}, $D$ can be transformed into the appropriate diagram to realize such crossing changes by region crossing changes. 
\end{proof}

\noindent For Eulerian spatial graphs, the following theorem follows from Corollary \ref{rcc-proper-g} and Lemma \ref{lemma-four}. 

\begin{theorem}
An Eulerian spatial graph $S$ is completely splittable by region crossing changes if and only if $S$ is proper. 
\label{E-splittable}
\end{theorem}

\begin{proof}
Let $S$ be an Eulerian spatial graph. 
If $S$ is proper, then $S$ is completely splittable by region crossing changes by Lemma \ref{lemma-four}. 
If $S$ is not proper, any diagram of $S$ is not proper, and furthermore any diagram which is obtained from a diagram of $S$ by region crossing changes is not proper by Corollary \ref{rcc-proper-g}. 
Then $S$ is not completely splittable by region crossing changes by Lemma \ref{lemma-four}. 
\end{proof}

\noindent The following theorem also follows. 

\begin{theorem}
A spatial graph $S$ of an Eulerian planar graph is unknottable by region crossing changes if and only if $S$ is proper. 
\label{E-unknottable}
\end{theorem}

\begin{proof}
Let $D=D_1 \cup D_2 \cup \dots \cup D_r$ be a diagram of Eulerian planar graph $S$ of $r$ components. 
If $S$ is proper, then $D$ has a set $R_0$ of regions which makes $D$ completely splitted by region crossing changes by Lemma \ref{lemma-four}. 
Also, $D$ has a set $R_i$ of regions which makes $D_i$ unknotted by region crossing changes by Lemma \ref{conn-pl-E}. 
Hence, the symmetric difference $R$ of $R_0, R_1, R_2, \dots ,R_r$ makes $D$ unknotted by region crossing changes. 
We remark that some reducible crossings of $D_i$ may have different results of region crossing changes between $R_i$ and $R$, where a {\it reducible crossing} is a crossing such that the same region meets diagonally at the crossing. 
There is no problem in that case because the crossing informations at reducible crossings do not matter for unknottedness. 

If $S$ is not proper, then $S$ is not completely splittable and hence is not unknottable by region crossing changes. 
\end{proof}

\noindent We prove Theorems \ref{thm-unknottable} and \ref{thm-splittable}. 

\phantom{x}
\noindent {\bf Proof of Theorem \ref{thm-splittable}} \ 
It follows from Theorems \ref{nE-splittable} and \ref{E-splittable}. 
\hfill$\Box$

\phantom{x}
\noindent {\bf Proof of Theorem \ref{thm-unknottable}} \ 
It follows from Theorems \ref{nE-unknottable} and \ref{E-unknottable}. 
\hfill$\Box$

\section*{Acknowledgment}
The authors are very grateful to Ryo Nikkuni for helpful comments. 
The second author's work was partially supported by JSPS KAKENHI Grant Number 17K14239.

\end{document}